\documentclass[12pt]{article}
\usepackage{amssymb}
\usepackage{amsthm}
\usepackage{amsmath}
\usepackage{color}
\usepackage[T1]{fontenc}

\DeclareMathOperator{\Con}{Con}

\DeclareMathOperator{\Id}{Id}
\DeclareMathOperator{\BId}{\mathbf{Id}}

\DeclareMathOperator{\BCon}{\mathbf{Con}}

\newtheorem{definition}{Definition}

\newtheorem{theorem}[definition]{Theorem}
\newtheorem{remark}[definition]{Remark}
\newtheorem{example}[definition]{Example}
\newtheorem{corollary}[definition]{Corollary}
\title{Ideals of direct products of rings}
\author{Ivan~Chajda, G\"unther Eigenthaler and Helmut~L\"anger}
\date{}
\begin{document}
\footnotetext[1]{Support of the research of all three authors by \"OAD, project CZ~04/2017, support of the first and the third author by IGA, project P\v rF~2018~012, as well as support of the second and the third author by the Austrian Science Fund (FWF), project I~1923-N25, is gratefully acknowledged.}
\footnotetext[2]{Electronic version of an article published in Asian-European Journal of Mathematics, Vol.\ 11, No.\ 4, 2018, pages 1850094-1--1850094-6, DOI: 10.1142/S1793557118500948\;\textcopyright\;World Scientific Publishing Company, https://www.worldscientific.com/worldscinet/aejm}
\maketitle
\begin{abstract}
It is known that an ideal of a direct product of commutative unitary rings is directly decomposable into ideals of the corresponding factors. We show that this does not hold in general for commutative rings and we find necessary and sufficient conditions for direct decomposability of ideals. For varieties of commutative rings we derive a Mal'cev type condition characterizing direct decomposability of ideals and we determine explicitly all varieties satisfying this condition.
\end{abstract}



Recall that an ideal of a ring $\mathbf R=(R,+,\cdot)$ is a non-empty subset $I$ of $R$ such that if $a,b\in I$ then $a-b\in I$ and $ar,ra\in I$ for every $r\in R$. For other concepts used here the reader is referred to any monograph on rings, see e.g.\ \cite A and \cite L. Ideals play a crucial role in the theory of rings since the kernels of homomorphisms are ideals and rings can be factorized by means of ideals.

It is an elementary fact that in rings, ideals and congruences are in a one-to-one correspondence. Hence, if $\mathbf R=(R,+,\cdot)$ is a ring then the ideal lattice $\BId\mathbf R:=(\Id\mathbf R,\subseteq)$ and the congruence lattice $\BCon\mathbf R:=(\Con\mathbf R,\subseteq)$ are isomorphic. Hence $\BId\mathbf R$ is a modular bounded lattice with the least element $\{0\}$ and the greatest element $R$ where the supremum and infimum of ideals $I_1,I_2$ are given by $I_1\vee I_2=I_1+I_2$ and $I_1\wedge I_2=I_1\cap I_2$, respectively.

Having two rings $\mathbf R_1=(R_1,+,\cdot)$ and $\mathbf R_2=(R_2,+,\cdot)$, it is elementary that for $I_1\in\Id\mathbf R_1$ and $I_2\in\Id\mathbf R_2$ we have $I_1\times I_2\in\Id(\mathbf R_1\times\mathbf R_2)$. On the other hand, if $I\in\Id(\mathbf R_1\times\mathbf R_2)$ then there need not exist $I_1\in\Id\mathbf R_1$ and $I_2\in\Id\mathbf R_2$ with $I_1\times I_2=I$. If such ideals $I_1,I_2$ do not exist, then $I$ is called {\em skew}. Otherwise, $I$ will be called {\em directly decomposable}. For commutative rings $\mathbf R_1,\mathbf R_2$ we will derive conditions under which $\mathbf R_1\times\mathbf R_2$ has no skew ideals.

We say that a direct product of finitely many rings has {\em directly decomposable ideals} if every ideal of this product is a direct product of ideals of the corresponding factors.

For sets $M_1,M_2$ and $i\in\{1,2\}$ let $\pi_i$ denote the $i$-th projection from $M_1\times M_2$ onto $M_i$.

It is easy to see that an ideal $I$ of $\mathbf R_1\times\mathbf R_2$ is directly decomposable if and only if $\pi_1(I)\times\pi_2(I)=I$ which is equivalent to $\pi_1(I)\times\pi_2(I)\subseteq I$.

Direct decomposability of ideals in commutative rings was used by the first and the third author in their study of complementation in ideal lattices, see \cite{CL}.

Let $\mathbf K_4=(\{0,a,b,c\},+)$ denote the four-element Kleinian group whose operation table for $+$ looks as follows:
\[
\begin{array}{c|cccc}
+ & 0 & a & b & c \\
\hline
0 & 0 & a & b & c \\
a & a & 0 & c & b \\
b & b & c & 0 & a \\
c & c & b & a & 0
\end{array}
\]
We show an example of rings whose direct product contains a skew ideal.

\begin{example}\label{ex2}
Consider the zero-ring $\mathbf K=(K,+,\cdot)$ whose additive group is $\mathbf K_4$, i.e., $xy=0$ for all $x,y\in K$. Then the principal ideal $I(a,c)=\{(0,0),(a,c)\}$ of $\mathbf K\times\mathbf K$ generated by $(a,c)$ is skew since $I\neq\pi_1(I)\times\pi_2(I)$.
\end{example}

In what follows, we derive necessary and sufficient conditions under which an ideal of the direct product of rings is directly decomposable. For this purpose, we borrow the method developed by Fraser and Horn (\cite{FH}) for congruences.

\begin{theorem}\label{th1}
Let $\mathbf R_1=(R_1,+,\cdot)$ and $\mathbf R_2=(R_2,+,\cdot)$ be rings and $I\in\Id(\mathbf R_1\times\mathbf R_2)$. Then the following are equivalent:
\begin{enumerate}
\item[{\rm(i)}] $I$ is directly decomposable.
\item[{\rm(ii)}] $(R_1\times\{0\})\cap((\{0\}\times R_2)+I)\subseteq I$\text{ and} \\
$((R_1\times\{0\})+I)\cap(\{0\}\times R_2)\subseteq I$.
\item[{\rm(iii)}] If $(a,b)\in I$ then $(a,0),(0,b)\in I$.
\item[{\rm(iv)}] $((R_1\times\{0\})+I)\cap((\{0\}\times R_2)+I)=I$.
\end{enumerate}
\end{theorem}

\begin{proof}
$\text{}$ \\
(i) $\Rightarrow$ (ii): \\
If $I=I_1\times I_2$ then
\begin{align*}
(R_1\times\{0\})\cap((\{0\}\times R_2)+I) & =(R_1\times\{0\})\cap((\{0\}\times R_2)+(I_1\times I_2))= \\
                                          & =(R_1\times\{0\})\cap(I_1\times R_2)=I_1\times\{0\}\subseteq I, \\
((R_1\times\{0\})+I)\cap(\{0\}\times R_2) & =((R_1\times\{0\})+(I_1\times I_2))\cap(\{0\}\times R_2)= \\
                                          & =(R_1\times I_2)\cap(\{0\}\times R_2)=\{0\}\times I_2\subseteq I.
\end{align*}
(ii) $\Rightarrow$ (iii): \\
If $(a,b)\in I$ then
\begin{align*}
(a,0) & \in(R_1\times\{0\})\cap((\{0\}\times R_2)+I)\subseteq I, \\
(0,b) & \in((R_1\times\{0\})+I)\cap(\{0\}\times R_2)\subseteq I.
\end{align*}
(iii) $\Rightarrow$ (i): \\
If $(a,b)\in\pi_1(I)\times\pi_2(I)$ then there exists some $(c,d)\in R_1\times R_2$ with $(a,d),(c,b)\in I$, hence $(a,0),(0,b)\in I$ which shows $(a,b)=(a,0)+(0,b)\in I$. \\
(i) $\Rightarrow$ (iv): \\
If $I=I_1\times I_2$ then
\begin{align*}
& ((R_1\times\{0\})+I)\cap((\{0\}\times R_2)+I)= \\
& =((R_1\times\{0\})+(I_1\times I_2))\cap((\{0\}\times R_2)+(I_1\times I_2))=(R_1\times I_2)\cap(I_1\times R_2)= \\
& =I_1\times I_2=I.
\end{align*}
(iv) $\Rightarrow$ (ii): \\
This follows immediately.
\end{proof}

The following result is already known but it follows easily from the equivalence of (i) and (iii) in Theorem~\ref{th1}.

\begin{corollary}
If $\mathbf R_1$ and $\mathbf R_2$ are unitary rings then $I$ is directly decomposable since $(a,b)\in I$ implies $(a,0)=(a,b)(1,0)\in I$ and $(0,b)=(a,b)(0,1)\in I$ and hence {\rm(iii)} holds.
\end{corollary}

\begin{corollary}
Let $\mathbf R_1,\mathbf R_2$ be rings. Then $\BId(\mathbf R_1\times\mathbf R_2)$ is distributive if and only if every ideal of $\mathbf R_1\times\mathbf R_2$ is directly decomposable and $\BId\mathbf R_1,\BId\mathbf R_2$ are distributive.
\end{corollary}

\begin{proof}
If $\BId(\mathbf R_1\times\mathbf R_2)$ is distributive, then (ii) of Theorem~\ref{th1} is satisfied for every $I\in\Id(\mathbf R_1\times\mathbf R_2)$. Thus every ideal of $\mathbf R_1\times\mathbf R_2$ is directly decomposable. Moreover, for $i=1,2$ we have that $\BId\mathbf R_i$ is isomorphic to the principal filter of $\BId(\mathbf R_1\times\mathbf R_2)$ generated by the kernel of $\pi_i$. Therefore, $\BId\mathbf R_1$ and $\BId\mathbf R_2$ are distributive. Conversely, suppose that $\BId\mathbf R_1,\BId\mathbf R_2$ are distributive and every ideal of $\mathbf R_1\times\mathbf R_2$ is directly decomposable. Let $I,J\in\Id(\mathbf R_1\times\mathbf R_2)$. Then there exist $I_1,J_1\in\Id\mathbf R_1$ and $I_2,J_2\in\Id\mathbf R_2$ with $I_1\times I_2=I$ and $J_1\times J_2=J$. Now we have
\begin{align*}
  I\vee J & =(I_1\times I_2)+(J_1\times J_2)=(I_1+J_1)\times(I_2+J_2), \\
I\wedge J & =(I_1\times I_2)\cap(J_1\times J_2)=(I_1\cap J_1)\times(I_2\cap J_2).
\end{align*}
Hence, join and meet in $\BId(\mathbf R_1\times\mathbf R_2)$ are computed ``component-wise'' showing distributivity of $\BId(\mathbf R_1\times\mathbf R_2)$.
\end{proof}

Another application of Theorem~\ref{th1} is the following example:

\begin{example}\label{ex1}
If $\mathbf R_1=(R_1,+,\cdot)$ is a Boolean  ring {\rm(}i.e., $xx\approx x${\rm)} and $\mathbf R_2=(R_2,+,\cdot)$ a unitary ring then $\mathbf R_1\times\mathbf R_2$ has no skew ideals. This follows directly by {\rm(iii)} of Theorem~\ref{th1} since if $(a,b)\in I$ then $(a,0)=(a,b)(a,0)\in I$ and $(0,b)=(a,b)(0,1)\in I$.
\end{example}

Denote by $\mathbb Z$ the ring of integers. As usually, for $a\in\mathbb Z$ put $a\mathbb Z:=\{ax\mid x\in\mathbb Z\}$. Of course, for each $a\in\mathbb Z$, $a\mathbb Z$ is a commutative ring which is unitary only in case $a\in\{-1,0,1\}$. The next theorem shows that, in general, the rings $a\mathbb Z\times b\mathbb Z$ contain skew ideals. Hence, the rather exotic ring $\mathbf K\times\mathbf K$ from Example~\ref{ex2} is not the only commutative ring possessing skew ideals.

\begin{theorem}\label{th3}
Let $a,b,c,d\in\mathbb Z$ with $c|a$ and $d|b$ and consider the ideal $I(a,b)$ of $c\mathbb Z\times d\mathbb Z$ generated by $(a,b)$. Then $I(a,b)$ is directly decomposable if and only if either $a=0$ or $b=0$ or $\gcd(c,d)=1$.
\end{theorem}

\begin{proof}
Put $(a,b)\mathbb Z:=\{(ax,bx)\mid x\in\mathbb Z\}$. Obviously, $I(a,b)=(a,b)\mathbb Z+(ac\mathbb Z\times bd\mathbb Z)$. Since $\pi_1(I(a,b))=a\mathbb Z$ and $\pi_2(I(a,b))=b\mathbb Z$, direct decomposability of $I(a,b)$ is equivalent to $a\mathbb Z\times b\mathbb Z\subseteq I(a,b)$. If $a=0$ or $b=0$ then obviously $a\mathbb Z\times b\mathbb Z\subseteq I(a,b)$. Now assume $a,b\neq0$. If $\gcd(c,d)=1$ then there exist $e,f\in\mathbb Z$ with $ce+df=1$ and hence
\[
(ax,by)=(a,b)(dfx+cey)+((ac)(e(x-y)),(bd)(-f(x-y)))\in I(a,b)
\]
for all $x,y\in\mathbb Z$ proving direct decomposability of $I(a,b)$. Finally, assume $\gcd(c,d)\neq1$. Suppose $I(a,b)$ to be directly decomposable. Then $(a,0)\in I(a,b)$ and hence there exist $g,h,i\in\mathbb Z$ with $(a,0)=(ag+ach,bg+bdi)$. From this we conclude $g=-di$ and hence $ch-di=1$ whence $\gcd(c,d)=1$, a contradiction. Hence $I(a,b)$ is not directly decomposable in this case.
\end{proof}

Although the rings of the form $c\mathbb Z$ are rings of integers, Theorem~\ref{th3} yields the following result.

\begin{corollary}
If $c,d\in\mathbb Z$ and $\gcd(c,d)\neq1$ then $c\mathbb Z\times d\mathbb Z$ has skew ideals.
\end{corollary}

\begin{example}
The principal ideal
\[
I(2,2)=(2,2)\mathbb Z+(4\mathbb Z\times4\mathbb Z)=(4\mathbb Z\times4\mathbb Z)\cup((2+4\mathbb Z)\times(2+4\mathbb Z))
\]
of $2\mathbb Z\times2\mathbb Z$ generated by $(2,2)$ is skew since $(0,0),(2,2)\in I(2,2)$, but $(0,2)\notin I(2,2)$. This is in accordance with Theorem~\ref{th3}.
\end{example}

Using Theorem~\ref{th1} we can generalize the situation described in Example~\ref{ex1}. Namely, we can derive a Mal'cev type condition characterizing varieties of commutative rings whose ideals are directly decomposable.

\begin{theorem}\label{th2}
Let $\mathcal V$ be a variety of commutative rings. Then $\mathcal V$ has directly decomposable ideals if and only if there exists a unary term $t$ satisfying the identity $xt(x)\approx x$.
\end{theorem}

\begin{proof}
First assume $\mathcal V$ to have directly decomposable ideals. Consider the free commutative ring $\mathbf F(x)=(F(x),+,\cdot)$ with one free generator $x$ and the principal ideal
\begin{align*}
I(x,x) & =\{n(x,x)+(x,x)(r(x),s(x))\mid n\in\mathbb Z\text{ and }(r(x),s(x))\in F(x)\times F(x)\}= \\
       & =\{(nx+xr(x),nx+xs(x))\mid n\in\mathbb Z\text{ and }r(x),s(x)\in F(x)\}
\end{align*}
of $\mathbf F(x)\times\mathbf F(x)$ generated by $(x,x)$. Since $I(x,x)$ is directly decomposable and $(0,0),$ $(x,x)\in I(x,x)$, we have $(x,0)\in I(x,x)$. Hence there exist some $n\in\mathbb Z$ and $r(x),s(x)\in F(x)$ with $(nx+xr(x),nx+xs(x))=(x,0)$ which implies $xt(x)=x$ with $t(x):=r(x)-s(x)$. Conversely, assume there exists a unary term $t$ satisfying $xt(x)\approx x$. Let $\mathbf R_1=(R_1,+,\cdot),\mathbf R_2=(R_2,+,\cdot)\in\mathcal V$ and $I\in\Id(\mathbf R_1\times\mathbf R_2)$ and assume $(a,b)\in I$. Then $(a,0)=(a,b)(t(a),0)\in I$ and $(0,b)=(a,b)(0,t(b))\in I$. According to (iii) of Theorem~\ref{th1}, $I$ is directly decomposable.
\end{proof}

Using Theorem~\ref{th2}, we can explicitly describe all varieties of commutative rings having directly decomposable ideals.

\begin{corollary}
A variety of commutative rings has directly decomposable ideals if and only if it satisfies an identity of the form $\sum\limits_{i=2}^na_ix^i\approx x$ with $n\geq2$ and $a_2,\ldots,a_n\in\mathbb Z$.
\end{corollary}

\begin{proof}
This follows immediately from Theorem~\ref{th2} since the unary terms in a variety of commutative rings are exactly the terms of the form $\sum\limits_{i=1}^na_ix^i$ with $n\geq1$ and $a_1,\ldots,a_n\in\mathbb Z$.
\end{proof}

\begin{corollary}
The variety of Boolean rings has directly decomposable ideals since it satisfies the identity $x^2\approx x$.
\end{corollary}

However, we can also consider classes of commutative rings which do not form a variety. In this case we cannot apply Theorem~\ref{th2} in order to prove that rings belonging to such a class have directly decomposable ideals. A typical example is the following:

\begin{example}\label{ex3}
Let $\mathbf R=(R,+,\cdot)$ be a commutative ring, $\mathbf F=(F,+,\cdot)$ a field, $I\in\Id(\mathbf R\times\mathbf F)$ and $J:=\pi_1(I)$. If $I\subseteq R\times\{0\}$ then $I=J\times\{0\}$. Now assume $I\not\subseteq R\times\{0\}$. Then there exists some $(a,b)\in I$ with $b\neq0$. If $(c,d)\in J\times F$ then there exists some $e\in F$ with $(c,e)\in I$ and hence
\[
(c,d)=(c,e)+(a,b)(0,b^{-1}(d-e))\in I
\]
showing $I=J\times F$. Hence, $\mathbf R\times\mathbf F$ has directly decomposable ideals. By induction we obtain that $\mathbf R\times\mathbf F_1\times\cdots\times\mathbf F_n$ has directly decomposable ideals if $n\geq1$ and $\mathbf F_1,\ldots,\mathbf F_n$ are fields.
\end{example}

In particular, we can consider the case where $\mathbf R$ denotes the commutative ring $\mathbf K$ from Example~\ref{ex2}. Then $\mathbf K\times\mathbf F$ has directly decomposable ideals despite the fact that $\mathbf K\times\mathbf K$ does not have this property. Similarly, if $c$ and $d$ are integers satisfying $\gcd(c,d)\neq1$ then $(c\mathbb Z\times d\mathbb Z)\times\mathbf F$, $c\mathbb Z\times\mathbf F$ and $d\mathbb Z\times\mathbf F$ have directly decomposable ideals though $c\mathbb Z\times d\mathbb Z$ does not have this property.

\begin{remark}
Note that Example~\ref{ex3} remains valid in case that $\mathbf R$ is not commutative.
\end{remark}

{\bf Acknowledgement.} The authors thank the referee for his valuable suggestions which increased the quality of the paper.

Authors' addresses:

Ivan Chajda \\
Palack\'y University Olomouc \\
Faculty of Science \\
Department of Algebra and Geometry \\
17.\ listopadu 12 \\
771 46 Olomouc \\
Czech Republic \\
ivan.chajda@upol.cz

G\"unther Eigenthaler \\
TU Wien \\
Faculty of Mathematics and Geoinformation \\
Institute of Discrete Mathematics and Geometry \\
Wiedner Hauptstra\ss e 8-10 \\
1040 Vienna \\
Austria \\
guenther.eigenthaler@tuwien.ac.at

Helmut L\"anger \\
TU Wien \\
Faculty of Mathematics and Geoinformation \\
Institute of Discrete Mathematics and Geometry \\
Wiedner Hauptstra\ss e 8-10 \\
1040 Vienna \\
Austria, and \\
Palack\'y University Olomouc \\
Faculty of Science \\
Department of Algebra and Geometry \\
17.\ listopadu 12 \\
771 46 Olomouc \\
Czech Republic \\
helmut.laenger@tuwien.ac.at
\end{document}